\documentclass[11pt]{amsart}

\usepackage{epigamath}

\usepackage[english]{babel}

\numberwithin{equation}{section}

\usepackage{enumitem}
\usepackage{tikz-cd}
\usepackage{bm}
\usepackage{verbatim}
\usepackage{longtable}

\newtheorem{theorem}{Theorem}[section]
\newtheorem{corollary}[theorem]{Corollary}
\newtheorem{lemma}[theorem]{Lemma}
\newtheorem{proposition}[theorem]{Proposition}

\theoremstyle{definition}

\theoremstyle{remark}
\newtheorem{remark}[theorem]{Remark}

\newcommand{\half}{\frac{1}{2}}
\newcommand{\CC} {{\mathbb C}}          
\newcommand{\RR} {{\mathbb R}}		
\newcommand{\ZZ} {{\mathbb Z}}		
\newcommand{\QQ} {{\mathbb Q}}		
\newcommand{\PP}{\mathbb{P}}
\newcommand{\HH}{\mathbb{H}}
\newcommand{\LL}{\mathbb{L}}
\newcommand{\X}{\mathfrak{X}}

\newcommand{\bir}{\mathsf{bir}}
\newcommand{\reg}{\mathsf{reg}}
\newcommand{\Ell}{\mathsf{Ell}}
\newcommand{\Ellqy}{\operatorname{Ell}_{q,y}}
\newcommand{\chibar}{\overline{\chi}}
\newcommand{\mdata}{\mathfrak{m}}

\renewcommand{\O}{\mathcal{O}}
\renewcommand{\top}{\,\mathsf{t}}
\newcommand{\varp}{\mathsf{p}}
\newcommand{\varq}{\mathsf{q}}
\newcommand{\varr}{\mathsf{r}}
\newcommand{\vara}{\mathsf{a}}

\newcommand{\Zcal}{\mathcal{Z}}
\newcommand{\mvec}{\bm{m}}
\newcommand{\zetavec}{\bm{\zeta }}
\newcommand{\rhovec}{\bm{\rho }}
\newcommand{\deltavec}{\bm{\delta  }}
\newcommand{\muvec}{\bm{\mu  }}
\newcommand{\dvec}{\bm{d }}
\newcommand{\uvec}{\bm{u }}
\newcommand{\vvec}{\bm{v }}
\newcommand{\wvec}{\bm{w }}

\newcommand{\Sym}{\operatorname{Sym}}

\newcommand{\Hilb}{\operatorname{Hilb}}
\newcommand{\SU}{\operatorname{SU}}
\newcommand{\SO}{\operatorname{SO}}

\EpigaVolumeYear{6}{2022} \EpigaArticleNr{6} \ReceivedOn{December 16,
2020}
\InFinalFormOn{January 11, 2022}
\AcceptedOn{January 27, 2022}

\title{$G$-fixed Hilbert schemes on $K3$ surfaces, modular forms, and
eta products}
\titlemark{$G$-fixed Hilbert schemes on $K3$ surfaces}

\author{Jim~Bryan}
\address{
Department of Mathematics, University of British Columbia, Room 121, 1984 Mathematics Road, Vancouver, B.C., Canada V6T 1Z2
}
\email{jbryan@math.ubc.ca}

\author{\'{A}d\'{a}m~Gyenge}
\address{Mathematical Institute, University of Oxford, Andrew Wiles Building, Woodstock Road, Oxford, OX2 6GG, United Kingdom
}
\email{Adam.Gyenge@maths.ox.ac.uk}

\authormark{J.~Bryan and \'A.~Gyenge}

\AbstractInEnglish{Let $X$ be a complex $K3$ surface with an effective action of a group
$G$ which preserves the holomorphic symplectic form. Let 
\[
Z_{X,G}(q) = \sum_{n=0}^{\infty} e\left(\Hilb^{n}(X)^{G} \right)\, q^{n-1}
\]
be the generating function for the Euler characteristics of the
Hilbert schemes of $G$-invariant length $n$ subschemes. We show that
its reciprocal, $Z_{X,G}(q)^{-1}$ is the Fourier expansion of a
modular cusp form of weight $\half e(X/G)$ for the congruence subgroup
$\Gamma_{0}(|G|)$. We give an explicit formula for $Z_{X,G}$ in terms
of the Dedekind eta function for all 82 possible $(X,G)$.  We extend
our results to various refinements of the Euler characteristic, namely
the Elliptic genus, the $\chi_{y}$ genus, and the motivic class. As a
byproduct of our method, we prove a result which is of independent
interest: it establishes an eta product identity for a certain shifted
theta function of the root lattice of a simply laced root system.}



\begin{document}



\maketitle

\begin{prelims}

\DisplayAbstractInEnglish

\bigskip

\DisplayKeyWords

\medskip

\DisplayMSCclass







\end{prelims}


\newpage

\setcounter{tocdepth}{1}

\tableofcontents


\section{Introduction}

Let $X$ be a complex $K3$ surface with an effective action of a group
$G$ which preserves the holomorphic symplectic form. Mukai showed that
such $G$ are precisely the subgroups of the Mathieu group
$M_{23}\subset M_{24}$ such that the induced action on the set
$\{1,\dots ,24 \}$ has at least five orbits
\cite{mukai1988finite}. Xiao classified all possible actions into
82 possible topological types of the quotient $X/G$ \cite{xiao1996galois}.

The \emph{$G$-fixed Hilbert scheme}\footnote{Some authors call this
the $G$-equivariant Hilbert scheme or the $G$-invariant Hilbert
scheme.} of $X$ parameterizes $G$-invariant length $n$ subschemes
$Z\subset X$. It can be identified with the $G$-fixed point locus in
the Hilbert scheme of points:
\[
\Hilb^{n}(X)^{G} \subset \Hilb^{n}(X).
\]

We define the corresponding \emph{$G$-fixed partition function} of
$X$ by
\[
Z_{X,G}(q) = \sum_{n=0}^{\infty} e\left(\Hilb^{n}(X)^{G} \right) q^{n-1} 
\]
where $e(-)$ is the topological Euler characteristic.

Throughout this paper we set
\[
q=\exp\left(2\pi i \tau  \right)
\]
so that we may regard $Z_{X,G}$ as a function of $\tau \in \HH$ where
$\HH$ is the upper half-plane.

\subsection{The Main Results.}
Our main result is the following:

\begin{theorem}
\label{thm:main} The function $Z_{X,G}(q)^{-1}$ is a modular cusp
form\footnote{By cusp form, we mean that the order of vanishing at
$q=0$ is at least 1. Modular forms of half integral weight transform with
respect to a multiplier system. We refer to \cite{kohler2011eta} for
definitions.} of weight $\half e(X/G)$ for the congruence subgroup
$\Gamma_{0}(|G|)$.
\end{theorem}

Our theorem specializes in the case where $G$ is the trivial group to
a famous result of G\"ottsche \cite{gottsche1990betti}. The case where
$G$ is a cyclic group was proved in \cite{bryan2018chl}. An analogous
result for the case where $X$ is an Abelian surface acted on
symplectically by a finite group $G$ has been recently proven by
Pietromonaco \cite{Pietromonaco-GHilbA}.


We give an explicit formula for $Z_{X,G}(q)$ in terms of the
Dedekind eta function
\[
\eta (\tau ) = q^{\frac{1}{24}}\prod_{n=1}^{\infty} (1-q^{n})
\]
as follows. Let $p_{1},\dots ,p_{r}$ be the singular points of $X/G$
and let $G_{1},\dots ,G_{r}$ be the corresponding stabilizer subgroups
of $G$. The singular points are necessarily of ADE type: they are
locally given by $\CC^{2}/G_{i}$ where $G_{i}\subset \SU(2)$. Finite
subgroups of $\SU(2)$ have an ADE classification and we let
$\Delta_{1},\dots ,\Delta_{r}$ denote the corresponding ADE root
systems.

For any finite subgroup $G_{\Delta}\subset \SU(2)$ with associated root
system $\Delta$ we define the \emph{local $G_{\Delta }$-fixed
partition function} by
\[
Z_{\Delta} (q) = \sum_{n=0}^{\infty}
e\left(\Hilb^{n}(\CC^{2})^{G_{\Delta}} \right) \, q^{n-\frac{1}{24}} .
\]
The main geometric result we prove is the following. 
\begin{theorem}\label{thm: formula for local series (in intro)}
The local partition function for $\Delta$ of type $A_{n}$ is given by 
\[
Z_{A_{n}}(q) = \frac{1}{\eta (\tau )}
\]
and for type $D_{n}$ and $E_{n}$ by
\[
Z_{\Delta}(q) = \frac{\eta^{2}(2\tau )\eta (4E\tau )}{\eta (\tau )\eta
(2E\tau )\eta (2F\tau )\eta (2V\tau )}
\]
where $(E,F,V)$ are given by:
\[
(E,F,V) = \begin{cases}
(n-2,2,n-2),\\
(6,4,4), \\
(12,8,6),  \\
(30,20,12), 
\end{cases}
\text{if}\quad \begin{array}{l}
\Delta =D_{n}\\ \Delta =E_{6}\\ \Delta =E_{7}\\ \Delta =E_{8}
\end{array} 
\]
\end{theorem}

\begin{remark}
For $\Delta$ of type $D_{n}$ or $E_{n}$, the group $H =
G_{\Delta}/\{\pm 1 \}\subset \SO(3)$ is the symmetry group of a
polyhedral decomposition of $S^{2}\cong \PP^{1}$ into isomorphic
regular spherical polygons. Then $E$, $F$, and $V$ are the number of
edges, faces, and vertices of the polyhedron.  The key idea in proving
the above theorem is to show that $\Hilb (\CC^{2})^{G_{\Delta}}$ is
deformation equivalent to $\Hilb (Y)^{H}$ where
$Y=\operatorname{Tot}(K_{\PP^{1}})$ is the minimal resolution of
$\CC^{2}/\{\pm 1 \}$ (see Section \ref{sec: proof of formula for local
series}).
\end{remark}

Using the work of Nakajima, we will also prove in Lemma~\ref{lem: local
series as theta/eta} that
\[
Z_{\Delta}(q) =\frac{\theta_{\Delta}(\tau)}{\eta (k\tau )^{n+1}}
\]
where 
\[
\theta_{\Delta}(\tau ) = \sum_{\mvec \in M_{\Delta}}
q^{\frac{k}{2}\left(\mvec +\frac{1}{k}\zetavec |\mvec
+\frac{1}{k}\zetavec \right)}
\]
is a shifted theta function for $M_{\Delta} $, the root lattice of
$\Delta$. Here $n$ is the rank of the root system, $k=|G_{\Delta}|$,
and $\zetavec$ is dual to the longest root (see Section~\ref{sec:
local partition functions} and Equation~\eqref{eqn: defn of shifted theta
function} for details).

Theorem~\ref{thm: formula for local series (in intro)} then yields an
eta product identity for the theta function $\theta_{\Delta}(\tau )$
reminiscent of the MacDonald identities:
\begin{theorem}\label{thm: eta product for theta function}
The shifted theta function $\theta _{\Delta}(\tau )$ defined above
$($cf. \S~\ref{sec: local partition functions} and Equation~\eqref{eqn: defn of
shifted theta function}$)$ is given by an eta product as follows:
\[
\theta_{A_{n}} (\tau ) = \frac{\eta^{n+1} ((n+1)\tau )}{\eta (\tau )}
\]
for $\Delta$ of type $A_{n}$ and  
\[
\theta_{\Delta}(\tau ) = \frac{\eta^{2}(2\tau )\,\eta ^{n+2}(4E\tau
)}{\eta (\tau )\, \eta (2E\tau )\,\eta (2F\tau )\,\eta (2V\tau )}
\]
for $\Delta$ of type $D_{n}$ or $E_{n}$, where $E,F,V$ are as in Theorem~\ref{thm: formula for local series (in intro)}.
\end{theorem}

\begin{remark}
Kac found that the Macdonald identities could be interpreted in terms
of the character formula for highest weight representations of
Kac-Moody algebras (\emph{cf.} \cite[\S~10]{kac1994infinite}). It would be
very interesting to find such an interpretation of the new identities in
Theorem~\ref{thm: eta product for theta function}.
\end{remark}

The 82 possible collections of ADE root systems $\Delta_{1},\dots
,\Delta_{r}$ associated to $(X,G)$ a $K3$ surface with a symplectic
$G$ action, are given in Appendix~\ref{app:tableeta},
Table~\ref{table: list of eta products}. We let $k=|G|$,
$k_{i}=|G_{i}|$, and
\[
a = e(X/G) - r=\frac{24}{k}-\sum_{i=1}^{r} \frac{1}{k_{i}}.
\]

The global series $Z_{X,G}(q)$ can be expressed as a product of local
contributions (and thus via Theorem~\ref{thm: formula for local series
(in intro)} as an explicit eta product) by our next result:
\begin{theorem}\label{thm: eta product formula for Z}
With the above notation we have
\[
Z_{X,G}(q) = \eta^{-a}(k\tau )\prod_{i=1}^{r}
Z_{\Delta_{i}}\left(\frac{k\tau}{k_{i}} \right).
\]
\end{theorem}

Theorem~\ref{thm:main} then immediately follows from Theorem~\ref{thm:
Zorbifold formula} and Theorem~\ref{thm: eta product formula for Z}
using the formulas for the weight and level of an eta product given in
\cite[\S~2.1]{kohler2011eta}.

In Appendix~\ref{app:tableeta}, Table~\ref{table: list of eta
products} we have listed explicitly the eta product of the modular form
$Z_{X,G}(q)^{-1}$ for all 82 possible cases of $(X,G)$.

\subsection{Consequences of the Main Results.}  Having obtained
explicit eta product expressions for $Z_{X,G}(q)$ allows us to make
several observational corollaries:

\begin{corollary}\label{cor: if G is a subgp of E then Zinv is a Hecke
eigenform} If $G$ is a finite subgroup of an elliptic curve $E$,
i.e. $G$ is isomorphic to a product of one or two cyclic groups, then
$Z_{X,G}(q)^{-1}$ is a Hecke eigenform. In Table~\ref{table: list of
eta products} these are the 13 cases having Xiao number in the set
$\{0,1,2,3,4,5,7,8,11,14,15,19,25 \}$. Moreover, in each of these
cases, the dimension of the Hecke eigenspace is one.
\end{corollary}

We remark that in these cases, we may form a Calabi-Yau threefold
called a CHL model by taking the free group quotient
\[
(X\times E)/G
\]
Then the partition function $Z_{X,G}(q)$ gives the Donaldson-Thomas
invariants of $(X\times E)/G$ in curve classes which have degree zero
over $X/G$ (\emph{cf.} \cite{bryan2018chl}).

\begin{remark}\label{rem: weight 3 Hecke case corresponds to arithmetic}
Hecke eigenforms of weight 3 arise in the arithmetic of $K3$ surfaces:
if $X$ is a $K3$ surface defined over $\QQ$ and has 
$\rho (X)=\operatorname{rk}\operatorname{NS(X)}=20$, then there is a weight 3 Hecke
eigenform
\[
f_{X}(q) = \sum_{n=1}^{\infty} a_{n}q^{n}
\]
such that for almost all primes $p$, $a_{p}$ is the trace of the
$p$-th Frobenius morphism acting on $H^{2}(X)/NS(X)$. There are four
cases where $Z_{X,G}(q)^{-1}$ is a weight three Hecke eigenform and
they correspond to the cases where $G$ is $\ZZ _7$, $\ZZ _8$, $\ZZ
_2\times \ZZ _6$, or $\ZZ _4\times \ZZ _4$ (numbered 8, 14, 19, 25 on
Table~\ref{table: list of eta products}). If $X$ admits a symplectic
$G$ action for one of these four groups, then we may take $X$ to be
defined over $\QQ$, have $\rho (X)=20$, and then remarkably
\[
Z_{X,G}(q)^{-1} = f_{X}(q).
\]
Indeed, in each of these cases, we may take $X$ to be elliptically
fibered over $\PP^{1}$ and have $G$ as its group of sections (thus
giving rise to the symplectic $G$ action). Moreover, $X$ is then the
universal curve over the modular curve parameterizing $(E,G)$, an
elliptic curve $E$ with a subgroup $G\subset E$. We thank Shuai Wang
and Noam Elkies for noticing and elucidating this phenomenon.
\end{remark}

For any eta product expression of a modular form, one may easily
compute the order of vanishing (or pole) at any of the cusps
\cite[Corollary~2.2]{kohler2011eta}. Performing this computation on the 82
cases yields the following:

\begin{corollary}\label{cor: vanishing at cusps}
The modular form $Z_{X,G}(q)^{-1}$ always vanishes with order 1 at the cusps
$i\infty$ and $0$. Moreover,
$Z_{X,G}(q)^{-1}$ is holomorphic at all cusps except for the two cases with
Xiao number 38 or 69, which have poles at the cusps $1/2$ and $1/8$
respectively. These are precisely the cases where $X/G$ has two
singularities of type $E_{6}$.
\end{corollary}

\begin{remark}\label{rem: enumerative interpretation}
The integers $e\left(\Hilb^{n}(X)^{G} \right)$ should have enumerative
significance: they can be interpreted as virtual counts of
$G$-invariant curves, whose quotient is rational, in a complete linear
series of dimension $n$ on $X$. This generalizes the famous Yau-Zaslow
formula \cite{Yau-Zaslow} in the case where $G$ is the trivial
group. The precise nature between the virtual count and the actual
count is expected to be subtle for the case of general $G$. This has
been recently explored in \cite{Zhan-2019-counting-curves-on-K3} and
also in the case of $G$ acting on an Abelian surface in
\cite{Pietromonaco-GHilbA}.
\end{remark}

\subsection{Refinements of the Euler Characteristic.}
We can extend our results to various refinements of the Euler
characteristic, namely the elliptic genus, the $\chi_{y}$ genus, and
the motivic class. These refinements all stem from the next
result which we prove in Section~\ref{sec: proof of thm about Zbir}. Let 
\[
Z^{\bir}_{X,G}(q) = \sum_{n=0}^{\infty} [\Hilb^{n}(X)^{G}]_{\bir} \,
\, q^{n-1}
\]
be a formal series whose coefficients we regard as birational
equivalence classes of projective hyperkahler manifolds. Such equivalence
classes form a semi-ring under disjoint union and Cartesian
product.

\begin{theorem}\label{thm: Formula for Zbir}
Let $Y$ be the minimal resolution of $X/G$, then
\[
Z^{\bir}_{X,G} (q) = Z^{\bir}_{Y}(q^{k})\cdot Z_{X,G}(q)\cdot \Delta (k\tau )
\]
where $k=|G|$, $\Delta (\tau ) = \eta (\tau )^{24}$, and we have
suppressed the trivial group from the notation in the series
$Z^{\bir}_{Y}(q^{k})$.
\end{theorem}

A famous theorem of Huybrechts \cite[Theorem~4.6]{Huybrechts} asserts that
birational projective hyperkahler manifolds are deformation
equivalent. Moreover, combining Huybrechts' theorem with
\cite[Proposition~3.21]{Nicaise-Shinder} it follows that birational projective
hyperkahler manifolds are equal in $K_{0}(\operatorname{Var}_{\CC})$,
the Grothendieck group of varieties.

Thus we may specialize the series $Z^{\bir}_{X,G}(q)$ to Elliptic
genus, motivic class, and $\chi_{y}$ genus since these are all well
defined on birational equivalence classes of projective hyperkahler
manifolds. These specializations are all well known for the series
$Z^{\bir}_{Y}$ and hence we easily get the following corollaries.

\begin{corollary}\label{cor: Zell formula}
Let $Q=\exp\left(2\pi i\sigma \right)$, $q=\exp\left(2\pi i \tau 
\right)$, $y=\exp\left(2\pi i z
\right)$, and let
\[
Z^{\Ell}_{X,G}(Q,q,y) = \sum_{n=0}^{\infty} \Ellqy
\left(\Hilb^{n}(X)^{G} \right) Q^{n-1}
\]
where $\Ellqy (-)$ is elliptic genus. Then
\[
Z^{\Ell}_{X,G}(Q,q,y) = \frac{\phi_{10,1}(\tau  ,z)}{\chi_{10}(k\sigma
,\tau ,z)} \cdot Z_{X,G}(q)\cdot \Delta (k\tau )
\]
where $\phi_{10,1}(q ,y)$ is the unique Jacobi cusp form of weight
10 and index 1 and $\chi_{10}(\sigma ,\tau ,z)$ is Igusa's genus 2
Siegel cusp form of weight 10. 
\end{corollary}
We refer the reader to \cite[\S~5, \S~6, and
Equation~6.9.8]{Pietromonaco_2018} for definitions of $\Ellqy$, $\phi_{10,1}$, $\chi_{10}$,
and the formula for the elliptic genera of $\Hilb^{n}(Y)$.

A further specialization of particular interest is the (normalized)
$\chi_{y}$ genus. Let 
\begin{align*}
\chibar_{-y}(M) &= y^{-\half \dim
M}\,\, \chi_{-y}(M)\\
& =  y^{-\half \dim
M} \,\sum_{p,q} (-1)^{p+q} y^{q} \, \dim H^{p,q}(M)
\end{align*}
and we note that $\chibar_{-y}(M) = \Ellqy
(M)|_{q=0}$.

\begin{corollary}\label{cor: Zchiy formula}
Let
\[
Z^{\chibar}_{X,G}(q,y) = \sum_{n=0}^{\infty}
\chibar_{-y}\left(\Hilb^{n}(X)^{G} \right) q^{n-1} .
\]
Then
\[
Z^{\chibar}_{X,G}(q,y) = y^{-1}(1-y)^{2} \frac{Z_{X,G}(q)}{\phi_{-2,1}(q^{k},y)}
\]
where $\phi_{-2,1}$ is the unique weak Jacobi form of weight $-2$
and index $1$. In particular,
\[
y^{-1}(1-y)^{2} Z^{\chibar}_{X,G}(q,y)^{-1} =\frac{
\phi_{-2,1}(q^{k},y)}{Z_{X,G}(q)} 
\]
is a Jacobi form of index 1 and weight 
\[
\half e(X/G)-2 = 10-\half \sum_{i=1}^{r} \operatorname{rank} \Delta_{i}
\]
for the congruence subgroup $\Gamma_{1}(k)$. 
\end{corollary}

We note that for $G$ cyclic, the series
$Z_{X,G}(q)/\phi_{-2,1}(q^{k},y) $ is the leading coefficient in the
expansion of the Donaldson-Thomas partition function of $(X\times
E)/G$ in the variable tracking the curve class in $X$ (see
\cite[Theorem~0.1]{bryan2018chl}).

We also get a formula for the motivic classes of the $G$-fixed Hilbert schemes:

\begin{corollary}\label{cor: ZK0 formula}
Let 
\[
Z^{K_{0}}_{X,G}(q) = \sum_{n=0}^{\infty} [\Hilb^{n}(X)^{G}]_{K_{0}}\,\, q^{n-1}
\]
where $ [\Hilb^{n}(X)^{G}]_{K_{0}}\in K_{0}(\operatorname{Var}_{\CC})$
denotes the motivic class of the $G$-fixed Hilbert scheme. Then
\[
Z^{K_{0}}_{X,G}(q) = q^{-1}\cdot \prod_{m=1}^{\infty} \left(1-\LL^{m-1}
q^{km} \right)^{-[Y]}  \cdot Z_{X,G}(q)\cdot \Delta (k\tau )
\]
where $\LL  = [\mathbb{A}^{1}_{\CC}]\in
K_{0}(\operatorname{Var}_{\CC})$.
\end{corollary}
We refer the reader to \cite{GZ-L-MH-power} for the meaning
of $[Y]$ in the exponent and the formula for the motivic class of
$\Hilb^{n}(Y)$. The above series has further specializations giving formulas for the  Hodge polynomials and
Poincare polynomials of the $G$-fixed Hilbert schemes.

\subsection{Structure of paper. }  In Section~\ref{sec: the global
series} we express the global partition function in terms of the local
partition functions and deduce Theorem~\ref{thm: eta product formula
for Z}. In Section~\ref{sec: proof of formula for local series} we
prove our main geometric result Theorem~\ref{thm: formula for local
series (in intro)} which gives the eta product expression of the local
partition functions. In Section~\ref{sec: local partition functions}
we express the local partition functions in terms of certain theta
functions and thus prove our Theorem~\ref{thm: eta product for theta
function} which gives us the new theta function identities. In
Section~\ref{sec: proof of thm about Zbir} we obtain the enhanced
result of Theorem~\ref{thm: Formula for Zbir} on the partition
function birational equivalence class of the $G$-fixed Hilbert
schemes. Appendix~\ref{sec: another strange formula} contains a proof
of a root theoretic identity we need and Appendix~\ref{app:tableeta}
contains a table listing the modular form $Z^{-1}_{X,G}$ in all 82
topological types of symplectic actions on a $K3$ surface.

\subsection*{Acknowledgements.} The authors warmly thank Jenny Bryan,
Noam Elkies, Federico Amadio Guidi, Georg Oberdieck, Ken Ono, Stephen
Pietromonaco, Bal\'azs Szendr\H{o}i, Shuai Wang, and Alex Weekes for
helpful comments and/or technical help.  We would also like to
thank the anonymous referee for helping us to fix and greatly simplify
the proof of Proposition~\ref{prop: Hilb(X,m0,m1) =
Hilb(Y,m0-(m0-m1)^2)}.

\section{The global partition function}\label{sec: the global series}

As in the introduction, let $X$ be a $K3$ surface with a symplectic
action of a finite group $G$.  Recall that $p_{1},\dotsc ,p_{r}\in
X/G$ are the singular points of $X/G$ with corresponding stabilizer
subgroups $G_{i}\subset G$ of order $k_{i}$ and ADE type
$\Delta_{i}$. Let $\{x_{i}^{1},\dotsc ,x_{i}^{k/k_{i}} \}$ be the
orbit of $G$ in $X$ corresponding to the point $p_{i}$ (recall that
$k=|G|$).  We may stratify $\Hilb (X)^{G}$ according to the orbit
types of subscheme as follows.

Let $Z\subset X$ be a $G$-invariant subscheme of length $nk$ whose
support lies on free orbits. Then $Z$ determines and is determined by
a length $n$ subscheme of 
\[
(X/G)^{o}  = X/G\setminus \{p_{1},\dotsc ,p_{r} \},
\]
\emph{i.e.} a point
in $\Hilb^{n}((X/G)^{o})$.

On the other hand, suppose $Z\subset X$ is a $G$-invariant subscheme
of length $\frac{nk}{k_{i}}$ supported on the orbit
$\{x_{i}^{1},\dotsc ,x_{i}^{k/k_{i}} \}$. Then $Z$ determines and is
determined by the length $n$ component of $Z$ supported on a formal
neighborhood of one of the points, say $x_{i}^{1}$. Choosing a
$G_{i}$-equivariant isomorphism of the formal neighborhood of
$x_{i}^{1}$ in $X$ with the formal neighborhood of the origin in
$\CC^{2}$, we see that $Z$ determines and is determined by a point in
$\Hilb_{0}^{n}(\CC^{2})^{G_{i}}$, where $\Hilb_{0}^{n}(\CC^{2})\subset
\Hilb^{n}(\CC^{2})$ is the punctual Hilbert scheme parameterizing
subschemes supported on a formal neighborhood of the origin in
$\CC^{2}$.

By decomposing an arbitrary $G$-invariant subscheme into components of
the above types, we obtain a stratification of $\Hilb (X)^{G}$ into
strata which are given by products of $\Hilb ((X/G)^{o})$ and
$\Hilb_{0}(\CC^{2})^{G_{1}},\dotsc ,\Hilb_{0}(\CC^{2})^{G_{r}}$. Then
using the fact that Euler characteristic is additive under
stratifications and multiplicative under products, we arrive at the
following equation of generating functions:
\begin{equation}\label{eqn: stratification formula for sum e(hilb(X)G)}
\sum_{n=0}^{\infty} e\left(\Hilb^{n}(X)^{G} \right)\, q^{n}
=\left(\sum_{n=0}^{\infty} e\left(\Hilb^{n}((X/G)^{o}) \right)\,
q^{kn} \right) \cdot \prod_{i=1}^{r}\left( \sum_{n=0}^{\infty}
e\left(\Hilb_{0}^{n}(\CC^{2})^{G_{i}} \right) \, q^{\frac{nk}{k_{i}}}
\right) .
\end{equation}

As in the introduction, let $a = e(X/G)-r=e\left((X/G)^{o}
\right)$. Then by G\"ottsche's formula \cite{gottsche1990betti},
\[
\sum_{n=0}^{\infty} e\left(\Hilb^{n}((X/G)^{0} \right) q^{kn} =
\prod_{m=1}^{\infty} (1-q^{km})^{-a} = q^{\frac{ak}{24}} \cdot \eta (k\tau )^{-a}. 
\]

We also note that $e\left(\Hilb_{0}^{n}(\CC^{2})^{G_{i}}
\right)=e\left(\Hilb^{n}(\CC^{2})^{G_{i}} \right)$ since the natural
$\CC^{*}$ action on both $\Hilb_{0}^{n}(\CC^{2})^{G_{i}}$ and
$\Hilb^{n}(\CC^{2})^{G_{i}}$ have the same fixed points. Thus we may
write
\[
\sum_{n=0}^{\infty} e\left(\Hilb_{0}^{n}(\CC^{2})^{G_{i}} \right) \,
q^{\frac{nk}{k_{i}}} = \sum_{n=0}^{\infty} e\left(\Hilb^{n}(\CC^{2})^{G_{i}} \right) \,
q^{\frac{nk}{k_{i}}} = q^{\frac{k}{24k_{i}}} \cdot Z_{\Delta_{i}} \left(\frac{k\tau}{k_{i}}
\right) .
\]

Multiplying Equation~\eqref{eqn: stratification formula for sum
e(hilb(X)G)} by $q^{-1}$ and substituting the above formulas, we find
that

\[
Z_{X,G}(q) = q^{-1 +\frac{ak}{24} + \sum \frac{k}{24k_{i}} } \cdot 
\eta (k\tau )^{-a}\cdot 
\prod_{i=1}^{r}Z_{\Delta_{i}}\left(\frac{k\tau}{k_{i}} \right) .
\]

>From the following Euler characteristic calculation, we see that the exponent of $q$ in the above equation is zero:
\[
24 = e(X) = e\left(X-\cup_{i=1}^{r} \{x_{i}^{1},\dotsc
,x_{i}^{k/k_{i}} \} \right) + \sum_{i=1}^{r} \frac{k}{k_{i}} 
= k \cdot e\left((X/G)^{o} \right) + \sum_{i=1}^{r} \frac{k}{k_{i}} 
= k\cdot a + \sum_{i=1}^{r}\frac{k}{k_{i}} 
\]

This completes the proof of Theorem~\ref{thm: eta product formula for
Z}.  \qed

\section{The local partition function}\label{sec: proof of formula for local series}

Recall that the local partition function is defined by
\[
Z_{\Delta} (q) = \sum_{n=0}^{\infty}
e\left(\Hilb^{n}(\CC^{2})^{G_{\Delta}} \right) \, q^{n-\frac{1}{24}} 
\]
where $G_{\Delta}\subset \SU(2)$ is the finite subgroup corresponding
to the ADE root system $\Delta$. In this section, we prove
Theorem~\ref{thm: formula for local series (in intro)} which provides
an explicit formula for $Z_{\Delta}(q)$ in terms of the Dedekind eta
function.  We regard this as the main geometric result of this paper.

\subsection{Proof of Theorem~\ref{thm: formula for local series (in
intro)} in the $A_{n}$ case.}\label{subsec: proof of An case of
local series}

We wish to prove
\[
Z_{A_{n}}(q) = \frac{1}{\eta (\tau )}
\]
which is equivalent to the statement
\[
\sum_{m=0}^{\infty} e\left(\Hilb^{m} (\CC^{2})^{\ZZ /(n+1)} \right)
\,q^{m} = \prod_{m=1}^{\infty} (1-q^{m})^{-1}.
\]
The action of $\ZZ /(n+1)$ on $\CC^{2}$ commutes with the action of
$\CC^{*}\times \CC^{*}$ on $\CC^{2}$ and consequently, the Euler
characteristics on the left hand side may be computed by counting
the $\CC^{*}\times \CC^{*}$-fixed subschemes, namely those given by
monomial ideals. Such subschemes of length $m$ have a well known
bijection with integer partitions of $m$, whose generating function is
given by the right hand side.\qed 

\subsection{Proof of Theorem~\ref{thm: formula for local series (in
intro)} in the $D_{n}$ and $E_{n}$ cases.}\label{subsec: proof of Dn
and En cases of local series}

Our proof of Theorem~\ref{thm: formula for local series (in intro)} in
the $D_{n}$ and $E_{n}$ cases uses a trick exploiting the fact that
the Hilbert schemes of the stack $\X =[\CC^{2}/\{\pm 1 \}]$ and the
Hilbert schemes of the space $Y=\operatorname{Tot}(K_{\PP^{1}})$ can
both be realized as moduli spaces of quiver representations of the
$A_{1}$ Nakajima quiver variety.

Let $G\subset \SU(2)$ be a subgroup where the corresponding root system
$\Delta$ is of $D$ or $E$ type. Then $\{\pm 1 \}\subset G$ and let
$H\subset \SO(3)$ be the quotient
\[
H=G/\{\pm 1 \}.
\]
The induced action of $H$ on $\PP^{1}\cong S^{2}$ is by
rotations. Indeed, $H$ is the symmetry group of a regular polyhedral
decomposition of $S^{2}$ which is given by the platonic solids in the
$E_{n}$ case and the decomposition into two hemispherical $(n-2)$-gons
in the $D_{n}$ case. $H$ is generated by rotations of order $\varp$,
$\varq$, $\varr$, obtained by rotating about the center of an edge, a
face, or a vertex respectively. The group $H$ has the following presentation:
\[
H=\{\left\langle a,b,c \right\rangle :\quad a^{\varp} = b^{\varq} =
c^{\varr} = abc=1 \}.
\]

Let $M=|H|$ be the order of $H$ and let $E,F,V$ be the number of
edges, faces, and vertices respectively. Then
\[
M=\varp E = \varq F = \varr V
\]
and since the stabilizer of an edge is always of order 2 we have $\varp =2$
and so $M=2E$. Then since $F+V-E=2$ we find
\[
E+F+V = 2 + M
\]

We summarize this information below:

\begin{center}
\begin{tabular}{|c|c|c|c|c|}
\hline
Type &	$H$ &	$M$&	$(\varp ,\varq ,\varr )$ &	$(E,F,V)$\\ \hline \hline 
$D_{n}$     & dihedral & $2n-2$ & $(2,n-2,2)$& $(n-1,2,n-1)$	\\ \hline
$E_{6}$     & tetrahedral & $12$ & $(2,3,3)$& $(6,4,4)$	\\ \hline
$E_{7}$     & octahedral & $24$ & $(2,3,4)$& $(12,8,6)$	\\ \hline
$E_{8}$     & icosohedral & $60$ & $(2,3,5)$& $(30,20,12)$	\\ \hline
\end{tabular}
\end{center}

\bigskip

Now let $\X$ be the quotient stack 
\[
\X=[\CC^{2}/ \{\pm 1 \}]
\]
and let
\[
Y\cong \operatorname{Tot}(K_{\PP^{1}})
\]
be the minimal resolution of the singular space $X=\CC^{2}/\{\pm 1
\}$.

The quotient stack $[\PP^{1}/H]$ has three stacky points with
stabilizers of order $\varp ,\varq ,\varr$, and consequently the stack
quotient $[Y/H]$ has three orbifold points locally of the form
$[\CC^{2}/\ZZ_{\vara }]$ for $\vara \in \{\varp ,\varq ,\varr \}$.

We observe that 
\[
 [\CC^{2}/G ] \cong [\X /H ]
\]
and consequently
\[
\Hilb^{n}(\CC^{2})^{G}  \cong \Hilb^{n}(\X )^{H} .
\]
The scheme $ \Hilb^{n}(\X )$ decomposes into components
$\Hilb^{m_{0},m_{1}}(\X )$ with $n=m_{0}+m_{1}$ where the
corresponding $\{\pm 1 \}$ invariant subschemes $Z\subset \CC^{2}$
have the property that as a $\{\pm 1 \}$-representation,
$H^{0}(\O_{Z})$ has $m_{0}$ copies of the trivial representation and
$m_{1}$ copies of the non-trivial representation.

\begin{proposition}\label{prop: Hilb(X,m0,m1) = Hilb(Y,m0-(m0-m1)^2)}
$\Hilb^{m_{0},m_{1}}(\X )^{H}$ is deformation equivalent to and hence
diffeomorphic to $\Hilb^{m_{0}-(m_{0}-m_{1})^{2}}(Y)^{H}$. In
particular
\[
e\left(\Hilb^{m_{0},m_{1}}(\X )^{H} \right)
=e\left(\Hilb^{m_{0}-(m_{0}-m_{1})^{2}}(Y)^{H} \right). 
\]
\end{proposition}

\begin{proof}

It is known that both $\Hilb^{m_{0},m_{1}}(\X )$ and $\Hilb^{n}(Y)$
can be realised as moduli spaces of quiver representations of the
$A_{1}$ Nakajima quiver variety with dimension vectors $(m_{0},m_{1})$
and $(n,n)$ respectively and framing vector $(1,0)$ for both (see
\cite[Proposition~5.2]{craw2021quot} and \cite{Kuznetsov}). By the
dimension formula for quiver varieties \cite[Equation~(2.6)]{Nakajima1994Duke}
(see also Equation~\eqref{eqn: dimension formula for quiver
varieties}), $\dim \Hilb^{m_{0},m_{1}}(\X )=2n$ where
$n=m_{0}-(m_{1}-m_{0})^{2}$. Consequently, both
$\Hilb^{m_{0},m_{1}}(\X )$ and $\Hilb^{n}(Y)$ are crepant projective
resolutions of $\Sym^{n}(\CC^{2}/\{\pm 1 \})$. By \cite[Corollary
1.3]{bellamy2020birational}, any two projective crepant resolutions of
$\Sym^{n}(\CC^{2}/\{\pm 1 \})$ can be realized as moduli spaces of
quiver representations with the \emph{same} dimension and framing
vectors (namely $(n,n)$ and $(1,0)$) but with different stability
conditions. Then by Nakajima \cite[Corollary~4.2]{Nakajima1994Duke},
$\Hilb^{m_{0},m_{1}}(\X )$ is deformation equivalent to, and hence
diffeomorphic to $\Hilb^{m_{0}-(m_{1}-m_{0})^{2}}(Y)$. Moreover, as
$H$ acts symplectically on both resolutions, the hyperk\"ahler
reduction providing the deformation equivalence in the proof of
\cite[Corollary~4.2]{Nakajima1994Duke} can be performed $H$-equivariantly to
obtain a deformation equivalence between the $H$-fixed points of the
two resolutions.
\end{proof}

Let
\[
j=m_{1}-m_{0}, \quad n=m_{0}-(m_{0}-m_{1})^{2} 
\quad \text{so that}
\quad
m_{0}+m_{1} = 2n +j + 2j^{2}. 
\]
We then can compute:
\[
q^{\frac{1}{24}} Z_{\Delta}(q) = \sum_{m_{0},m_{1}= 0}^{\infty } 
e\left(\Hilb^{m_{0},m_{1}}(\X )^{H} \right) \, q^{m_{0}+m_{1}} 
= \sum_{j\in \ZZ}\,  \sum_{n=0}^{\infty} e\left(\Hilb^{n}(Y)^{H}
\right)\, q^{2n+j+2j^{2}}.
\]

The following identity follows easily from the Jacobi triple product
formula:
\[
\sum_{j\in \ZZ} q^{2j^{2}+j+\frac{1}{8}} = \frac{\eta^{2}(2\tau )}{\eta (\tau )}.
\]

Substituting this into the previous equation multiplied by
$q^{\frac{1}{8}}$ we find
\[
q^{\frac{1}{6}} Z_{\Delta}(q) =  \frac{\eta^{2}(2\tau )}{\eta (\tau )} \cdot
\sum_{n=0}^{\infty}  e\left(\Hilb^{n}(Y)^{H}
\right)\, q^{2n}.
\]

We can now compute the summation factor in the above equation by the
same method we used to compute the global series in Section~\ref{sec:
the global series}. Here we use the fact that the singularities of
$Y/H$ are all of type $A$ and we have already proven our formula for
the local series in the $A_{n}$ case. Indeed, the quotient $[Y/H]$ has
three stacky points with stabilizers $\ZZ_{\varp}$, $\ZZ_{\varq}$, and
$\ZZ_{\varr}$ and the complement of those points $(Y/H)^{o}$ has Euler
characteristic $-1$. Proceeding then by the same argument we used in
Section~\ref{sec: the global series} to get Equation~\eqref{eqn:
stratification formula for sum e(hilb(X)G)}, we obtain
\begin{align*}
\sum_{n=0}^{\infty}  e\left(\Hilb^{n}(Y)^{H}\right)\, q^{2n}&=
\left(\sum_{n=0}^{\infty} e\left(\Hilb^{n}\left((Y/H)^{o} \right)
\right) q^{2Mn} \right)  \cdot \prod_{\vara \in \{\varp ,\varq
,\varr  \}}\, \left( \sum_{n=0}^{\infty} e\left(\Hilb_{0}^{n}(\CC^{2})^{\ZZ_{\vara }} \right) q^{\frac{2Mn}{\vara }}  \right)\\
&=\prod_{m=1}^{\infty} \frac{\left(1-q^{2Mn} \right)}{\left(1-q^{\frac{2Mn}{\varp}} \right)\left(1-q^{\frac{2Mn}{\varq}} \right)\left(1-q^{\frac{2Mn}{\varr}} \right)}\\
&=\prod_{m=1}^{\infty} \frac{\left(1-q^{4En} \right)}{\left(1-q^{2En} \right)\left(1-q^{2Fn} \right)\left(1-q^{2Vn} \right)}\\
&=q^{\frac{1}{24}(-2E+2F+2V)}\cdot  \frac{\eta (4E\tau )}{\eta (2E\tau )\eta
(2F\tau )\eta (2V\tau )} \\
&=  \frac{q^{\frac{1}{6}}\,\eta (4E\tau )}{\eta (2E\tau )\eta
(2F\tau )\eta (2V\tau )}.
\end{align*}
Substituting into the previous equation and cancelling the factors of
$q^{\frac{1}{6}}$, we have thus proved
\[
Z_{\Delta}(q) =  \frac{\eta^{2}(2\tau )}{\eta (\tau )}\cdot \frac{\eta (4E\tau )}{\eta
(2E\tau )\eta (2F\tau )\eta (2V\tau )},
\]
which completes the proof of Theorem~\ref{thm: formula for local
series (in intro)} in the general case. \qed

\section{The local partition function as a theta function via Nakajima}\label{sec: local partition functions}

The local partition functions $Z_{\Delta}(q)$ considered in this paper
are obtained from a specialization of the partition functions of the
stack $[\CC^{2} /G_{\Delta}]$.  Using the work of Nakajima
\cite{nakajima2002geometric}, the partition function of the Euler
characteristics of the Hilbert scheme of points on the stack quotient
$[\CC^{2}/G_{\Delta}]$ was computed explicitly in
\cite{gyenge2015euler} in terms of the root data of $\Delta$.  We use
this to express $Z_{\Delta}(q)$ in terms of $\theta_{\Delta}(\tau )$,
a shifted theta function for the root lattice of $\Delta$. As a
byproduct we obtain an eta product formula for the associated shifted
theta function (Theorem~\ref{thm: eta product for theta function}).

A zero-dimensional substack $Z\subset [\CC^{2}/G_{\Delta}]$ may be
regarded as a $G_{\Delta}$ invariant, zero-dimensional subscheme of
$\CC^{2}$. Consequently, we may identify the Hilbert scheme of points
on the stack $[\CC^{2}/G_{\Delta}]$ with the $G_{\Delta}$ fixed locus
of the Hilbert scheme of points on $\CC^{2}$: 
\[
\Hilb \left([\CC^{2}/G_{\Delta}] \right) = \Hilb
(\CC^{2})^{G_{\Delta}} .
\]

This Hilbert scheme has components indexed by representations $\rho$
of $G_{\Delta}$ as follows
\begin{equation*}
\Hilb^{\rho} \left([\CC^{2}/G_{\Delta}] \right) = \left\{ Z\subset
\CC^{2}, \text{ $Z$ is $G_{\Delta}$ invariant and $H^{0}(\O_{Z})\cong
\rho $} \right\}.
\end{equation*}

Let $\{\rho_{0},\dotsc ,\rho_{n} \}$ be the irreducible
representations of $G_{\Delta}$ where $\rho_{0}$ is the trivial
representation. We note that $n$ is also the rank of $\Delta$. We
define
\[
Z_{[\CC^{2}/G_{\Delta}]} (q_{0},\dotsc ,q_{n}) = \sum_{m_{0},\dotsc
,m_{n}=0}^{\infty} e\left(\Hilb^{m_{0}\rho_{0}+\dotsb
+m_{n}\rho_{n}}([\CC^{2}/G_{\Delta}]) \right) q_{0}^{m_{0}}\dotsb
q_{n}^{m_{n}} .
\]

Recall that our local partition function $Z_{\Delta}(q)$ is defined by
\[
Z_{\Delta}(q) = \sum_{n=0}^{\infty}
e\left(\Hilb^{n}(\CC^{2})^{G_{\Delta}} \right) q^{n-\frac{1}{24}}. 
\]
We then readily see that
\[
Z_{\Delta}(q) = q^{\frac{-1}{24}}\cdot  Z_{[\CC^{2}/G_{\Delta}]}
(q_{0},\dotsc ,q_{n})|_{q_{i}=q^{d_{i}}}
\]
where
\[
d_{i} =\dim \rho_{i}.
\]

The following formula is given explicitly in \cite[Theorem~1.3]{gyenge2015euler}, but its content is already present in the work of Nakajima \cite{nakajima2002geometric}:
\begin{theorem} \label{thm: Zorbifold formula}
Let $C_{\Delta}$ be the Cartan matrix of the root system $\Delta$,
then 
\[
Z_{[\CC^{2}/G_{\Delta}]} (q_{0},\dotsc ,q_{n}) = \prod_{m=1}^{\infty}
(1-Q^{m})^{-n-1} \cdot \sum_{\mvec \in \ZZ^{n}} q_{1}^{m_{1}}\dotsb
q_{n}^{m_{n}} \cdot Q^{\half \mvec^{\top}\cdot C_{\Delta}\cdot \mvec}
\]
where $Q=q_{0}^{d_{0}}q_{1}^{d_{1}}\dotsb q_{n}^{d_{n}}$.
\end{theorem}
We note that under the specialization $q_{i}=q^{d_{i}}$, 
\[
Q=q^{d_{0}^{2}+\dotsb +d_{n}^{2}} = q^{k}
\]
where $k=|G|$ is the order of the group $G$.

We then obtain
\[
Z_{\Delta}(q) = q^{\frac{-1}{24}}\cdot
\prod_{m=1}^{\infty}(1-q^{km})^{-n-1} \cdot \sum_{\mvec \in \ZZ^{n}}
q^{\mvec^{\top}\cdot \dvec} \cdot q^{\frac{k}{2}\mvec^{\top}\cdot C_{\Delta}\cdot \mvec}
\]
where $\dvec =(d_{1},\dotsc ,d_{n})$.

Let $M_{\Delta}$ be the root lattice of $\Delta$ which we identify
with $\ZZ^{n}$ via the basis given by $\alpha_{1},\dotsc ,\alpha_{n}$,
the simple positive roots of $\Delta$. Under this identification, the
standard Weyl invariant bilinear form is given by
\[
(\uvec |\vvec ) = \uvec^{\top}\cdot C_{\Delta}\cdot \vvec 
\]
and $\dvec$ is identified with the longest root. 
We define
\[
\zetavec = C_{\Delta}^{-1} \cdot \dvec 
\]
so that 
\[
\mvec^{\top}\cdot \dvec = \mvec^{\top}\cdot C_{\Delta} \cdot \zetavec
= (\mvec |\zetavec ).
\]
We may then write
\begin{align*}
Z_{\Delta}(q)& = q^{\frac{-1}{24}}\cdot
\prod_{m=1}^{\infty}(1-q^{km})^{-n-1}\cdot \sum _{\mvec \in M_{\Delta}} q^{(\mvec |\zetavec )+\frac{k}{2}(\mvec |\mvec)}\\
&= q^{A} \cdot \left(q^{\frac{k}{24}}\prod_{m=1}^{\infty}(1-q^{km})
\right)^{-n-1} \cdot \sum_{\mvec \in M_{\Delta}} q^{\frac{k}{2}\left(\mvec +\frac{1}{k}\zetavec |\mvec +\frac{1}{k}\zetavec  \right)}\\
& = q^{A}\cdot  \eta (k\tau )^{-n-1}\cdot  \theta_{\Delta} (\tau )
\end{align*}
where 
\[
A \ = \ \frac{-1}{24} + \frac{k(n+1)}{24} - \frac{1}{2k}(\zetavec
|\zetavec )
\]
and $\theta_{\Delta}(\tau )$ is the shifted theta function:
\begin{equation}\label{eqn: defn of shifted theta function}
\theta_{\Delta}(\tau ) = \sum_{\mvec \in M_{\Delta}} q^{\frac{k}{2}\left(\mvec +\frac{1}{k}\zetavec |\mvec +\frac{1}{k}\zetavec  \right)}
\end{equation}
where as throughout this paper we have identified $q=\exp\left(2\pi i\tau  \right)$.

In Appendix~\ref{sec: another strange formula}, we will prove the
following formula which for $\Delta =A_{n}$ coincides with the
``strange formula'' of Freudenthal and de Vries
\cite{freudenthal1969linear}:
\[
 \frac{k(n+1)-1}{24}  = \frac{(\zetavec |\zetavec )}{2k}.
\]
It follows that $A=0$ and we obtain the following:
\begin{lemma}\label{lem: local series as theta/eta} The local series
$Z_{\Delta}(q)$ is given by
\[
Z_{\Delta}(q) = \frac{\theta_{\Delta}(\tau )}{\eta (k\tau )^{n+1}}. 
\]
\end{lemma}

\section{Proof of Theorem~\ref{thm: Formula for Zbir}}
\label{sec: proof of thm about Zbir}

Let $\Zcal  =[X/G]$ be the quotient stack of $X$ by $G$ and let $Y\to
X/G$ be the minimal resolution. 
The Hilbert scheme of zero dimensional substacks of $\Zcal$ is
naturally identified with the $G$-fixed Hilbert scheme of $X$:
\[
\Hilb (\Zcal )\cong \Hilb (X)^{G}.
\]
We emphasize that $\Hilb (\Zcal )$ is itself a scheme, not just a
stack, as the objects it parameterizes (substacks $V\subset \Zcal $)
do not have automorphisms (see \cite{Olsson-Starr} or
\cite[\S~2.3]{Bryan-Cadman-Young}). Components of $\Hilb (\Zcal )$ are
indexed by the numerical $K$-theory class of $\O_{Z}$ for $Z\subset
\Zcal$. The $K$-theory class of $\O_{Z}$ can be written in a basis for
$K$-theory as follows:
\[
[\O_{Z}] = n[\O_{p}] + \sum_{i=1}^{r} \sum_{j=1}^{n(i)}
m_{j}(i)[\O_{p_{i}}\otimes \rho_{j}(i)] 
\]
where $p\in \Zcal$ is a generic point and $p_{1},\dotsc ,p_{r}\in \Zcal$
are the orbifold points. The local group of $\Zcal$ at $p_{i}$ is
$G_{\Delta (i)}\subset \SU(2)$  and has corresponding root system
$\Delta (i)$ of rank $n(i)$, and has irreducible representations
$\rho_{0}(i),\rho_{1}(i),\dotsc ,\rho_{n(i)}(i)$ where $\rho_{0}(i)$
is the trivial representation. 
We note that we do not need to include $ [\O_{p_{i}}\otimes
\rho_{0}(i)]$ in our basis for $K$-theory because of the following
relation in  $K$-theory which holds for all $i$:

\begin{equation}\label{eqn: Op = O0timesrhoreg}
[\O_{p}] = [\O_{p_{i}}\otimes \rho_{\reg}(i)]
\end{equation}
where $\rho_{\reg}(i)$ is the regular representation of
$G_{\Delta(i)}$.

We abbreviate the data $\left\{m_{j}(i) \right\}$ appearing in the
$K$-theory class above by the symbol $\mdata$ and we denote by
\[
\Hilb^{n,\mdata}(\Zcal ) \subset \Hilb (\Zcal )
\]
the corresponding component. Let
\[
D_{\mdata} = \sum_{i=1}^{r} \sum_{j=1}^{n(i)} \, m_{j}(i)\, E_{j}(i)
\]
where
$E_{1}(i),\dotsc ,E_{n(i)}(i)$ are
the exceptional curves over $p_{i}$. We can organize the data $\mdata = \left\{m_{j}(i) \right\}$ into
$\mvec (i)\in M_{\Delta (i)}$, \emph{i.e.}  the vectors in the root lattice
of $\Delta (i)$ having components $ m_{1}(i),\dotsc
,m_{n(i)}(i)$. Under this identification
\[
D_{\mdata}^{2} = -\sum_{i=1}^{r} \left(\mvec (i)|\mvec (i)
\right)_{\Delta (i)} 
\]
since the intersection form of the exceptional curves over $p_{i}$ is
the negative of the corresponding Cartan matrix $C_{\Delta (i)}$.

\begin{proposition}\label{prop: Hilb(Z) is birational to Hilb(Y)}
	$\Hilb^{n,\mdata}(\Zcal)$ is birational to $\Hilb^{n+\half
		D_{\mdata}^{2}}(Y)$. 
\end{proposition}

To prove this we will first need the following lemma.
\newpage

\begin{lemma}\label{lem:rigidsh}
\leavevmode
\begin{enumerate} \item Let $\Delta$ be a rank $n$ ADE root
system, let $M_{\Delta }$ be the corresponding root lattice, and let
$G_{\Delta } \subset \SU(2)$ the corresponding finite subgroup. To any
element $\mvec \in M_{\Delta }$ there is a unique rigid
\footnote{By
definition, a substack $Z$ is \emph{rigid} if it corresponds to an
isolated point in the Hilbert scheme. }
substack $Z_{\mvec}\subset [\mathbb{C}^2/G_{\Delta}]$ with $K$-theory
class
\[
\half (\mvec|\mvec)[\O_{p}] + \sum_{j=1}^{n}
m_{j}[\O_{0}\otimes \rho_{j}]
\]
where $p \in [\mathbb{C}^2/G]$ is a generic point.
	
		\item For every datum $\mdata$ there is a unique rigid
substack $Z_{\mdata} \subset \mathcal{Z}$ with $K$-theory
class \begin{gather*}\sum_{i=1}^r
\left(
\frac{1}{2}(\mvec(i)|\mvec(i))_{\Delta (i)}[\O_{p}] + \sum_{j=1}^{n(i)}
m_{j}(i)[\O_{p_i}\otimes \rho_{j}(i)]
 \right)\\
=-\frac{
D_{\mdata}^{2}}{2}[\O_{p}] +\sum_{j=1}^{r} \sum_{j=1}^{n(i)} m_{j}(i)[\O_{p_i}\otimes
\rho_{j}(i)] \end{gather*} where $p \in \Zcal$ is a generic point.
	\end{enumerate}
\end{lemma}
\begin{proof}
	Part (2) is implied by Part (1) since we can take the union of the rigid subschemes supported at the orbifold points $p_1,\dots,p_r \in \mathcal{Z}$. So we need only prove the local case.
	
To prove Part (1) we need to show that component of $\Hilb
([\CC^{2}/G_{\Delta}])$ corresponding to substacks with $K$-theory
class
\[
\half (\mvec|\mvec)[\O_{p}] + \sum_{j=1}^{n}
m_{j}[\O_{0}\otimes \rho_{j}]
\]
is a single isolated point. This component corresponds to the
coefficient of 
\[
Q^{\half(\mvec |\mvec )}\cdot  q_{1}^{m_{1}}\dotsb q_{n}^{m_{n}}
\]
in Theorem~\ref{thm: Zorbifold formula}. It follows immediately from
the formula in Theorem~\ref{thm: Zorbifold formula} that this
coefficient is 1, and thus to prove this component is a single point,
we need only prove that it has dimension 0.

By Equation~\eqref{eqn: Op
= O0timesrhoreg}, we have
\[
\half (\mvec|\mvec)[\O_{p}] + \sum_{j=1}^{n}
m_{j}[\O_{0}\otimes \rho_{j}] = 
\half (\mvec|\mvec)[\O_{0}\otimes \rho_{0}] + \sum_{j=1}^{n}
\left(\half (\mvec |\mvec )d_{j}+m_{j} \right)[\O_{0}\otimes \rho_{j}]
\]
and so the component in question is 
\[
\Hilb^{v_{0}\rho_{0}+\dotsb +v_{n}\rho_{n}}([\CC^{2} /G_{\Delta}])
\]
where
\[
\vvec =(v_{0},v_{1},\dotsc ,v_{n}) = \left(\half (\mvec |\mvec ),\half (\mvec |\mvec )d_{1}+m_{1},\dotsc ,\half (\mvec |\mvec )d_{n}+m_{n} \right).
\]
We define
\[
\deltavec = (1,d_{1},\dotsc ,d_{n})\text{ and }
\muvec = (0,m_{1},\dotsc ,m_{n})
\]
so that our $\vvec$ of interest may be written
\[
\vvec =\half (\mvec |\mvec )\deltavec +\muvec .
\]
Nakajima has shown \cite[\S~2]{nakajima2002geometric} that 
\[
\Hilb^{v_{0}\rho_{0}+\dotsb +v_{n}\rho_{n}}([\CC^{2} /G_{\Delta}]) = M(\vvec ,\wvec )
\]
where $\wvec = (1,0,\dotsc ,0)$ and $M(\vvec ,\wvec )$ is the Nakajima
quiver variety associated to the affine Dynkin diagram of $\Delta$
with framing vector $\wvec$ and dimension vector $\vvec$. By
\cite[Equation~(2.6)]{Nakajima1994Duke} we have 
\begin{equation}\label{eqn: dimension formula for quiver varieties}
\dim M(\vvec ,\wvec ) = 2\vvec \cdot \wvec -\left\langle \vvec ,\vvec \right\rangle
= 2v_{0} - \left\langle \vvec ,\vvec \right\rangle
=(\mvec |\mvec ) - \left\langle \vvec ,\vvec\right\rangle
\end{equation}
where $\left\langle \cdot , \cdot  \right\rangle$ is the inner product
given by the Cartan matrix associated to the affine Dynkin diagram.

We also have
\[
\left\langle \muvec ,\muvec \right\rangle = (\mvec |\mvec ),\quad
 \left\langle \deltavec ,\deltavec \right\rangle = 0, \quad 
\text{and }\ \left\langle \muvec ,\deltavec \right\rangle = 0.
\]
The first follows directly from our definitions, and the later two are
well known properties of the vector $\deltavec$.  Using the above we
compute the dimension of the Hilbert scheme of interest:
\begin{align*}
\dim M(\vvec ,\wvec ) &=
(\mvec| \mvec ) - \left\langle
\half ( \mvec| \mvec )\deltavec+\muvec ,\half ( \mvec| \mvec )\deltavec+\muvec  \right\rangle \\
& = (\mvec| \mvec )  -\frac{1}{4} ( \mvec|
\mvec)\langle \deltavec, \deltavec \rangle  - ( \mvec|
\mvec) \langle \muvec, \deltavec \rangle  - \langle \muvec, \muvec \rangle\\
& = 0.
\end{align*}
We thus can conclude that $\Hilb^{v_{0}\rho_{0}+\dotsb
+v_{n}\rho_{n}}([\CC^{2} /G_{\Delta}]) = M(\vvec ,\wvec )$ is a single
point which finishes the proof of the lemma.
\end{proof}

\begin{proof}[Proof of Proposition \ref{prop: Hilb(Z) is birational to
Hilb(Y)}] Let $U= \Zcal \setminus \{p_{1},\dotsc ,p_{r} \}$ be the
Zariski open part with trivial stabilizers. Let $V \subset Y$ be the complement of the
exceptional divisors.  Let furthermore $Z_{\mdata} \subset
\mathcal{Z}$ be the rigid substack corresponding to the $K$-theory
datum $\mdata$ provided by Lemma \ref{lem:rigidsh}.  The Zariski open
substack of $\mathrm{Hilb}^{n,\mdata}(\Zcal)$ parameterizing substacks
of $\Zcal$ of the form $P \cup Z_{\mdata}$ where $P$ is a colength
$n+\half D_{\mdata}^{2}$ subscheme of $U$ is isomorphic to
$\mathrm{Hilb}^{n+\half D_{\mdata}^{2}}(U)$. This is because
$Z_{\mdata}$ was rigid and it had the $K$-theory class
\[-\half D_{\mdata}^{2}[\O_{p}] + \sum_{j=1}^{n} m_{j}(i)[\O_{p_i}\otimes\rho_{j}(i)].\]
On the other hand, the Zariski open subset of
$\mathrm{Hilb}^{n+\half D_{\mdata}^{2}}(Y)$ parameterizing subschemes
supported on $V\subset Y$ is isomorphic to $\mathrm{Hilb}^{n+\half
D_{\mdata}^{2}}(V)$. Finally, $\mathrm{Hilb}^{n+\half
D_{\mdata}^{2}}(U)\cong \mathrm{Hilb}^{n+\half D_{\mdata}^{2}}(V)$
since $U$ and $V$ are canonically isomorphic.
\end{proof}

With Proposition~\ref{prop: Hilb(Z) is birational to Hilb(Y)}, we can
now prove Theorem~\ref{thm: Formula for Zbir}.  Using the
identification

\[
\Hilb(X)^{G} \cong  \Hilb (\Zcal )
\]
and identifying
discrete parameters we get
\[
Z^{\bir}_{X,G}(q) = \sum_{a=0}^{\infty} \left[\Hilb^{a}(X)^{G}
\right]_{\bir} \, \, q^{a-1} 
 = \sum_{n,\mdata} \left[\Hilb^{n,\mdata}(\Zcal ) \right]_{\bir } \,  q^{D(n,\mdata )-1}
\]
where (recalling that $d_{j}(i) =\dim \rho_{j}(i)$),
\[
D(n,\mdata ) = kn +\sum_{i=1}^{r}\frac{k}{k_{i}}\sum_{j=1}^{n(i)}
m_{j}(i) d_{j}(i).
\]

Let
\[
d=n+\half D_{\mdata}^{2} = n-\half \sum_{i=1}^{r} \left(\mvec (i)|\mvec (i)
\right)_{\Delta (i)} .
\]
Then
\[
Z^{\bir}_{X,G}(q) = \sum_{d=0}\left[\Hilb^{d}(Y) \right]_{\bir}
\prod_{i=1}^{r}\sum_{\mvec (i)\in M_{\Delta (i)}} q^{D(n,\mdata )-1} 
\]
with
\begin{align*}
D(n,\mdata )-1& = -1 + k\left(d+\half \sum_{i=1}^{r} \left(\mvec (i)|\mvec (i)
\right)_{\Delta (i)}  \right)
+\sum_{i=1}^{r}\frac{k}{k_{i}}\sum_{j=1}^{n(i)} m_{j}(i) d_{j}(i) \\
&=kd-1 +\frac{k}{2}\sum_{i=1}^{r} \left\{  \left(\mvec (i)|\mvec (i)
\right)_{\Delta (i)} +\frac{2}{k_{i}}  \left(\mvec (i)|\zetavec (i)
\right)_{\Delta (i)} \right\}
\end{align*}
where $\zetavec (i) \in M_{\Delta (i)}\otimes \QQ $ is as in
Section~\ref{sec: local partition functions}. 

Completing the square and using the formula
\[
\frac{1}{k_{i}^{2}}  \left(\zetavec (i)|\zetavec (i)
\right)_{\Delta (i)} = \frac{2}{k_{i}}\left(\frac{k_{i}(n(i)+1)-1}{24} \right),
\]
which follows from Lemma~\ref{lem: another strange formula}, we get 
\[
D(n,\mdata )-1 = kd-1  - \sum_{i=1}^{r}
\frac{k}{k_{i}}\left(\frac{k_{i}(n(i)+1)-1}{24} \right)
+\frac{k}{2}\sum_{i=1}^{r} \left(\mvec (i)+\frac{1}{k_{i}}\zetavec
(i)\,\,\Big{|}\,\,\mvec (i)+\frac{1}{k_{i}}\zetavec (i) \right)_{\Delta (i)} . 
\]
It then follows that
\[
Z^{\bir}_{X,G}(q) = q^{A} \sum_{d=0}^{\infty} \left[\Hilb^{d}(Y)
\right]_{\bir} \,\, q^{kd-k} \,\, \prod_{i=1}^{r} \sum_{\mvec (i)\in
M_{\Delta (i)}} q^{\frac{k}{2} \left(\mvec (i)+\frac{1}{k_{i}}\zetavec
(i)\,\,\Big{|}\,\,\mvec (i)+\frac{1}{k_{i}}\zetavec (i) \right)_{\Delta (i)} } 
\]
where
\[
A = k-1 -\frac{k}{24} \sum_{i=1}^{r} \left(n(i) + 1 - \frac{1}{k_{i}}
\right). 
\]
Since 
\begin{align*}
24 = e(Y) = e\left(X/G - \{p_{1},\dotsc ,p_{r} \} \right) +
\sum_{i=1}^{r} (n(i)+1)
&=\frac{1}{k} \left(24 - \sum_{i=1}^{r} \frac{k}{k_{i}} \right)
+\sum_{i=1}^{r} (n(i)+1)\\
&= \frac{24}{k} + \sum_{i=1}^{r} \left(n(i)+1 -\frac{1}{k_{i}} \right),
\end{align*}
we see that $A=0$.

Thus we have
\[
Z^{\bir}_{X,G} (q) = Z^{\bir}_{Y}(q^{k}) \prod_{i=1}^{r}
\theta_{\Delta (i)} \left(\frac{k\tau}{k_{i}} \right) 
= Z^{\bir}_{Y}(q^{k}) \prod_{i=1}^{r}
Z_{\Delta (i)} \left(\frac{k\tau}{k_{i}} \right)\eta (k\tau )^{n(i)+1} 
=Z_{Y}^{\bir} (q^{k})\cdot  \eta (k\tau )^{B}\cdot  Z_{X,G}(q) 
\]
where we used Theorem~\ref{thm: eta product for theta function},
Theorem~\ref{thm: eta product formula for Z}, and we set
\[
B=\frac{24}{k} + \sum_{i=1}^{r} \left(n(i)+1 -\frac{1}{k_{i}}
\right). 
\]
The previous equation which showed that $A=0$ also shows that
$B=24$. Then since $\Delta (\tau ) = \eta (\tau )^{24}$, we see that
Theorem~\ref{thm: Formula for Zbir} follows. \qed

\appendix
\renewcommand{\thesection}{\Alph{section}}
\setcounter{section}{0}

\section*{Appendix~A.\hspace{0.5em}Another Strange Formula}\label{sec: another strange
formula}
\addcontentsline{toc}{section}{Appendix~A.\hspace{0.5em}Another Strange Formula}
\refstepcounter{section}

We recall the notation from Section~\ref{sec: local partition
functions}. Let $\Delta$ be an ADE root system of rank $n$. Let
$\alpha_{1},\dotsc \alpha_{n}$ be a system of positive simple roots
and let
\[
\dvec  = \sum_{i=1}^{n} d_{i} \alpha_{i}
\]
be the largest root. Let $(\cdot |\cdot )$ be the Weyl invariant
bilinear form with $(\alpha_{i}|\alpha_{i})=2$ and let $\zetavec$ be
the dual vector to $\dvec$ in the sense that
\begin{equation}\label{eqn: d = sum (zeta|ai)ai}
\sum_{i=1}^{n} (\zetavec |\alpha_{i}) \alpha_{i} = \dvec .
\end{equation}
Let 
\begin{equation}\label{eqn: k=1+(zeta|d)}
k=1+\sum_{i=1}^{n}d_{i}^{2} =  1+(\zetavec |\dvec ). 
\end{equation}

The identity of the following lemma coincides with Freudenthal and de Vries's
``strange formula'' when $\Delta$ is $A_{n}$. 
\begin{lemma}\label{lem: another strange formula}
Let $k$, $n$, and $\zetavec$ be as above. Then,
\[
\frac{k(n+1)-1}{24} =\frac{(\zetavec |\zetavec )}{2k} . 
\]
\end{lemma}
\begin{proof}
\leavevmode

\textbf{The case of $\Delta =A_{n}$.---}
For any ADE root system we have $(\rhovec |\alpha )=1$ for all positive
roots where $\rhovec =\half \sum_{\alpha \in R^{+}}\alpha$ is half the
sum of the positive roots. Since for $A_{n}$, $d_{i}=1$, it follows
from Equation~\eqref{eqn: d = sum (zeta|ai)ai} that $\zetavec
=\rhovec$, and it follows from Equation~\eqref{eqn: k=1+(zeta|d)}
that $k=n+1=h$ is the Coxeter number. The lemma is then
\[
\frac{(n+1)^{2}-1}{24} = \frac{(\rhovec |\rhovec )}{2h}.
\]
Since the Lie algebra associated to $A_{n}$, namely $\mathfrak{sl}_{n+1}$,
has dimension $(n+1)^{2}-1$ and the Killing form satisfies $\kappa
(\cdot ,\cdot )=\frac{1}{2h}(\cdot |\cdot )$, the lemma may be
rewritten as
\[
\frac{\dim \mathfrak{sl}_{n+1}}{24} =\kappa (\rhovec ,\rhovec )
\]
which is Freudenthal and de Vries's ``Strange Formula''
\cite[\S~47.11]{freudenthal1969linear}.

\bigskip
\textbf{The case of $\Delta =D_{n}$.---} Let $e_{1},\dotsc ,e_{n}$ be the
standard orthonormal basis of $\RR^{n}$. Then the collection $\{\pm
e_{i} \pm e_{j}, i<j \}$ is a $D_{n}$ root system and we may take 
\[
\alpha_{i} = \begin{cases}
e_{i}-e_{i+1},\\
e_{n-1}+e_{n},
\end{cases} \text{if}\quad\begin{array}{l}
i=1,\dotsc ,n-1\\ i=n
\end{array}
\]
as a system of simple positive roots. Then the fundamental weights
$\omega_{i}$, which are defined by the condition $(\omega_{i}
|\alpha_{j}) = \delta_{ij} $, are given by \cite[Appendix C]{Knapp}
\[
\omega_{i} = \begin{cases}
e_{1}+\dotsb +e_{i},\\
\half (e_{1}+\dotsb +e_{n-1}-e_{n}),\\
\half (e_{1}+\dotsb +e_{n-1}+e_{n}),\\
\end{cases}\ \text{if}\quad\begin{array}{l}
i\leq n-2\\  i=n-1\\ i=n.
\end{array}
\]
Then since
\[
d_{i}=\begin{cases}
1 \\ 2
\end{cases}\text{if}\quad\begin{array}{l}i=1,n-1,n,\\
i=2,\dotsc , n-2,
\end{array} 
\]
we have
\[
\zetavec  = \omega_{1}+2\omega_{2}+2\omega_{3}+\dotsb +2\omega_{n-2}+\omega_{n-1}+\omega_{n}
= 2(n-2) e_{1} + \sum_{i=2}^{n-2}\left(2(n-1-i)+1 \right)\, e_{i}  +e_{n-1}
\]
and so
\[
(\zetavec |\zetavec ) = 4(n-2)^{2} + 1 + \sum_{i=2}^{n-2} \left(2(n-1-i)+1 \right)^{2}
= \frac{4}{3}n^{3}-4n^{2}-\frac{1}{3}n+6.
\]
Finally since $k=1+\sum_{i=1}^{n}d_{i} = 4(n-2)$ the lemma becomes
\[
\frac{4(n-2)(n+1)-1}{24} = \frac{ \frac{4}{3}n^{3}-4n^{2}-\frac{1}{3}n+6}{8(n-2)}
\]
which is readily verified.

\bigskip
\textbf{The case of $\Delta =E_{6},E_{7},E_{8}$.---} These three
individual cases are easily checked one by one.
\end{proof}


\section*{Apppendix~B.\hspace{0.5em}Table of eta products}\label{app:tableeta}
\addcontentsline{toc}{section}{Appendix~B.\hspace{0.5em}Table of eta products}
\refstepcounter{section}
\renewcommand{\arraystretch}{1.5}

The following table provides the list of the modular forms
$Z_{X,G}^{-1}$, expressed as eta products, for each of the 82 possible
symplectic actions of a group $G$ on a $K3$ surface $X$. Our numbering
matches Xiao's \cite{xiao1996galois} whose table we refer to for a
description of each group.
\begin{longtable}{|l|l|l|l|l|}
  \hline
$\# $ & $|G|$ & Singularities of $X/G$&  The modular form $Z_{X,G}^{-1}$ & Weight \\ 
  \hline
0 & 1 &  & $ \eta \left( \tau \right)   ^{24}$ & 12 \\ 
  1 & 2 & $8 A_{1}$ & $ \eta \left( 2\tau \right)   ^{8}  \eta \left( \tau \right)   ^{8}$ & 8 \\ 
  2 & 3 & $6 A_{2}$ & $ \eta \left( 3\tau \right)   ^{6}  \eta \left( \tau \right)   ^{6}$ & 6 \\ 
  3 & 4 & $12 A_{1}$ & $ \eta \left( 2\tau \right)   ^{12}$ & 6 \\ 
  4 & 4 & $2 A_{1} + 4 A_{3}$ & $ \eta \left( 4\tau \right)   ^{4}  \eta \left( 2\tau \right)   ^{2}  \eta \left( \tau \right)   ^{4}$ & 5 \\ 
  5 & 5 & $4 A_{4}$ & $ \eta \left( 5\tau \right)   ^{4}  \eta \left( \tau \right)   ^{4}$ & 4 \\ 
  6 & 6 & $8 A_{1} + 3 A_{2}$ & ${\frac {  \eta \left( 3\tau \right)   ^{8}  \eta \left( 2\tau \right)   ^{3}}{\eta \left( 6\tau \right) }}$ & 5 \\ 
  7 & 6 & $2 A_{1} + 2 A_{2} + 2 A_{5}$ & $ \eta \left( 6\tau \right)   ^{2}  \eta \left( 3\tau \right)   ^{2}  \eta \left( 2\tau \right)   ^{2} \mbox{}  \eta \left( \tau \right)   ^{2}$ & 4 \\ 
  8 & 7 & $3 A_{6}$ & $ \eta \left( 7\tau \right)   ^{3}  \eta \left( \tau \right)   ^{3}$ & 3 \\ 
  9 & 8 & $14 A_{1}$ & ${\frac {  \eta \left( 4\tau \right)   ^{14}}{  \eta \left( 8\tau \right)   ^{4}}}$ & 5 \\ 
  10 & 8 & $9 A_{1} + 2 A_{3}$ & ${\frac {  \eta \left( 4\tau \right)   ^{9}  \eta \left( 2\tau \right)   ^{2}}{  \eta \left( 8\tau \right)   ^{2}}}$ & 9/2 \\ 
  11 & 8 & $4 A_{1} + 4 A_{3}$ & $ \eta \left( 4\tau \right)   ^{4}  \eta \left( 2\tau \right)   ^{4}$ & 4 \\ 
  12 & 8 & $3 A_{3} + 2 D_{4}$ & ${\frac {  \eta \left( \tau \right)   ^{2}  \eta \left( 4\tau \right)   ^{6}}{\eta \left( 2\tau \right) }}$ & 7/2 \\ 
  13 & 8 & $ A_{1} + 4 D_{4}$ & ${\frac {  \eta \left( 4\tau \right)   ^{13}  \eta \left( \tau \right)   ^{4}}{  \eta \left( 8\tau \right)   ^{2} \mbox{}  \eta \left( 2\tau \right)   ^{8}}}$ & 7/2 \\ 
  14 & 8 & $ A_{1} +  A_{3} + 2 A_{7}$ & $ \eta \left( 8\tau \right)   ^{2}\eta \left( 4\tau \right) \eta \left( 2\tau \right)   \eta \left( \tau \right)   ^{2}$ & 3 \\ 
  15 & 9 & $8 A_{2}$ & $ \eta \left( 3\tau \right)   ^{8}$ & 4 \\ 
  16 & 10 & $8 A_{1} + 2 A_{4}$ & ${\frac {  \eta \left( 5\tau \right)   ^{8}  \eta \left( 2\tau \right)   ^{2}}{  \eta \left( 10\tau \right)   ^{2}}}$ & 4 \\ 
  17 & 12 & $4 A_{1} + 6 A_{2}$ & ${\frac {  \eta \left( 6\tau \right)   ^{4}  \eta \left( 4\tau \right)   ^{6}}{  \eta \left( 12\tau \right)   ^{2}}}$ & 4 \\ 
  18 & 12 & $9 A_{1} +  A_{2} +  A_{5}$ & ${\frac {  \eta \left( 6\tau \right)   ^{9}\eta \left( 4\tau \right) \eta \left( 2\tau \right) }{  \eta \left( 12\tau \right)   ^{3}}}$ & 4 \\ 
  19 & 12 & $3 A_{1} + 3 A_{5}$ & $ \eta \left( 6\tau \right)   ^{3}  \eta \left( 2\tau \right)   ^{3}$ & 3 \\ 
  20 & 12 & $ A_{2} + 2 A_{3} + 2 D_{5}$ & ${\frac {  \eta \left( 4\tau \right)   ^{3}  \eta \left( 3\tau \right)   ^{2}  \eta \left( \tau \right)   ^{2} \mbox{}  \eta \left( 6\tau \right)   ^{4}}{\eta \left( 12\tau \right)   \eta \left( 2\tau \right)   ^{4}}}$ & 3 \\ 
  21 & 16 & $15 A_{1}$ & ${\frac {  \eta \left( 8\tau \right)   ^{15}}{  \eta \left( 16\tau \right)   ^{6}}}$ & 9/2 \\ 
  22 & 16 & $10 A_{1} + 2 A_{3}$ & ${\frac {  \eta \left( 8\tau \right)   ^{10}  \eta \left( 4\tau \right)   ^{2}}{  \eta \left( 16\tau \right)   ^{4}}}$ & 4 \\ 
  23 & 16 & $5 A_{1} + 4 A_{3}$ & ${\frac {  \eta \left( 8\tau \right)   ^{5}  \eta \left( 4\tau \right)   ^{4}}{  \eta \left( 16\tau \right)   ^{2}}}$ & 7/2 \\ 
  24 & 16 & $6 A_{1} +  A_{3} + 2 D_{4}$ & ${\frac {  \eta \left( 8\tau \right)   ^{12}  \eta \left( 2\tau \right)   ^{2}}{  \eta \left( 16\tau \right)   ^{4} \mbox{}  \eta \left( 4\tau \right)   ^{3}}}$ & 7/2 \\ 
  25 & 16 & $6 A_{3}$ & $ \eta \left( 4\tau \right)   ^{6}$ & 3 \\ 
  26 & 16 & $4 A_{1} +  A_{3} +  A_{7} +  D_{4}$ & ${\frac {  \eta \left( 8\tau \right)   ^{7}  \eta \left( 2\tau \right)   ^{2}}{  \eta \left( 16\tau \right)   ^{2} \mbox{}\eta \left( 4\tau \right) }}$ & 3 \\ 
  27 & 16 & $2 A_{1} + 4 D_{4}$ & ${\frac {  \eta \left( 8\tau \right)   ^{14}  \eta \left( 2\tau \right)   ^{4}}{  \eta \left( 4\tau \right)   ^{8} \mbox{}  \eta \left( 16\tau \right)   ^{4}}}$ & 3 \\ 
  28 & 16 & $2 A_{1} +  A_{3} + 2 A_{7}$ & $ \eta \left( 8\tau \right)   ^{2}\eta \left( 4\tau \right)   \eta \left( 2\tau \right)   ^{2}$ & 5/2 \\ 
  29 & 16 & $ A_{3} +  D_{4} + 2 D_{6}$ & ${\frac {\eta \left( 4\tau \right)   \eta \left( 8\tau \right)   ^{7}  \eta \left( \tau \right)   ^{2}}{  \eta \left( 16\tau \right)   ^{2} \mbox{}  \eta \left( 2\tau \right)   ^{3}}}$ & 5/2 \\ 
  30 & 18 & $8 A_{1} + 4 A_{2}$ & ${\frac {  \eta \left( 9\tau \right)   ^{8}  \eta \left( 6\tau \right)   ^{4}}{  \eta \left( 18\tau \right)   ^{4}}}$ & 4 \\ 
  31 & 18 & $2 A_{1} + 3 A_{2} + 2 A_{5}$ & ${\frac {  \eta \left( 9\tau \right)   ^{2}  \eta \left( 6\tau \right)   ^{3}  \eta \left( 3\tau \right)   ^{2} \mbox{}}{\eta \left( 18\tau \right) }}$ & 3 \\ 
  32 & 20 & $2 A_{1} + 4 A_{3} +  A_{4}$ & ${\frac {  \eta \left( 10\tau \right)   ^{2}  \eta \left( 5\tau \right)   ^{4}\eta \left( 4\tau \right) }{\eta \left( 20 \mbox{}\tau \right) }}$ & 3 \\ 
  33 & 21 & $6 A_{2} +  A_{6}$ & ${\frac {  \eta \left( 7\tau \right)   ^{6}\eta \left( 3\tau \right) }{\eta \left( 21\tau \right) }}$ & 3 \\ 
  34 & 24 & $5 A_{1} + 3 A_{2} + 2 A_{3}$ & ${\frac {  \eta \left( 12\tau \right)   ^{5}  \eta \left( 8\tau \right)   ^{3}  \eta \left( 6\tau \right)   ^{2} \mbox{}}{  \eta \left( 24\tau \right)   ^{3}}}$ & 7/2 \\ 
  35 & 24 & $4 A_{1} + 2 A_{2} + 2 A_{5}$ & ${\frac {  \eta \left( 12\tau \right)   ^{4}  \eta \left( 8\tau \right)   ^{2}  \eta \left( 4\tau \right)   ^{2} \mbox{}}{  \eta \left( 24\tau \right)   ^{2}}}$ & 3 \\ 
  36 & 24 & $5 A_{1} +  A_{3} +  A_{5} +  D_{5}$ & ${\frac {  \eta \left( 12\tau \right)   ^{7}\eta \left( 6\tau \right) \eta \left( 2\tau \right) \eta \left( 8\tau \right) }{  \eta \left( 24\tau \right)   ^{3} \mbox{}\eta \left( 4\tau \right) }}$ & 3 \\ 
  37 & 24 & $2 A_{2} +  A_{5} +  D_{4} +  E_{6}$ & ${\frac {  \eta \left( 8\tau \right)   ^{4}\eta \left( 4\tau \right) \eta \left( 3\tau \right)   \eta \left( 12\tau \right)   ^{4} \mbox{}\eta \left( \tau \right) }{  \eta \left( 6\tau \right)   ^{2}  \eta \left( 24\tau \right)   ^{2}  \eta \left( 2\tau \right)   ^{2}}}$ & 5/2 \\ 
  38 & 24 & $2 A_{2} +  A_{3} + 2 E_{6}$ & ${\frac {  \eta \left( 8\tau \right)   ^{6}\eta \left( 6\tau \right)   \eta \left( \tau \right)   ^{2}  \eta \left( 12\tau \right)   ^{2} \mbox{}}{  \eta \left( 2\tau \right)   ^{4}  \eta \left( 24\tau \right)   ^{2}}}$ & 5/2 \\ 
  39 & 32 & $8 A_{1} + 3 A_{3}$ & ${\frac {  \eta \left( 16\tau \right)   ^{8}  \eta \left( 8\tau \right)   ^{3}}{  \eta \left( 32\tau \right)   ^{4}}}$ & 7/2 \\ 
  40 & 32 & $9 A_{1} + 2 D_{4}$ & ${\frac {  \eta \left( 16\tau \right)   ^{15}  \eta \left( 4\tau \right)   ^{2}}{  \eta \left( 32\tau \right)   ^{6} \mbox{}  \eta \left( 8\tau \right)   ^{4}}}$ & 7/2 \\ 
  41 & 32 & $3 A_{1} + 5 A_{3}$ & ${\frac {  \eta \left( 16\tau \right)   ^{3}  \eta \left( 8\tau \right)   ^{5}}{  \eta \left( 32\tau \right)   ^{2}}}$ & 3 \\ 
  42 & 32 & $4 A_{1} + 2 A_{3} + 2 D_{4}$ & ${\frac {  \eta \left( 16\tau \right)   ^{10}  \eta \left( 4\tau \right)   ^{2}}{  \eta \left( 32\tau \right)   ^{4} \mbox{}  \eta \left( 8\tau \right)   ^{2}}}$ & 3 \\ 
  43 & 32 & $5 A_{1} + 2 A_{7}$ & ${\frac {  \eta \left( 16\tau \right)   ^{5}  \eta \left( 4\tau \right)   ^{2}}{  \eta \left( 32\tau \right)   ^{2}}}$ & 5/2 \\ 
  44 & 32 & $2 A_{1} + 2 A_{3} +  A_{7} +  D_{4}$ & ${\frac {  \eta \left( 16\tau \right)   ^{5}  \eta \left( 4\tau \right)   ^{2}}{  \eta \left( 32\tau \right)   ^{2}}}$ & 5/2 \\ 
  45 & 32 & $3 A_{1} +  D_{4} + 2 D_{6}$ & ${\frac {  \eta \left( 16\tau \right)   ^{10}  \eta \left( 2\tau \right)   ^{2}}{  \eta \left( 32\tau \right)   ^{4} \mbox{}  \eta \left( 4\tau \right)   ^{3}}}$ & 5/2 \\ 
  46 & 36 & $2 A_{1} + 2 A_{2} + 4 A_{3}$ & ${\frac {  \eta \left( 18\tau \right)   ^{2}  \eta \left( 12\tau \right)   ^{2}  \eta \left( 9\tau \right)   ^{4} \mbox{}}{  \eta \left( 36\tau \right)   ^{2}}}$ & 3 \\ 
  47 & 36 & $ A_{1} + 6 A_{2} +  A_{5}$ & ${\frac {\eta \left( 18\tau \right)   \eta \left( 12\tau \right)   ^{6}\eta \left( 6\tau \right) }{  \eta \left( 36 \mbox{}\tau \right)   ^{2}}}$ & 3 \\ 
  48 & 36 & $6 A_{1} +  A_{2} + 2 A_{5}$ & ${\frac {  \eta \left( 18\tau \right)   ^{6}\eta \left( 12\tau \right)   \eta \left( 6\tau \right)   ^{2}}{  \eta \left( 36 \mbox{}\tau \right)   ^{3}}}$ & 3 \\ 
  49 & 48 & $5 A_{1} + 6 A_{2}$ & ${\frac {  \eta \left( 24\tau \right)   ^{5}  \eta \left( 16\tau \right)   ^{6}}{  \eta \left( 48\tau \right)   ^{4}}}$ & 7/2 \\ 
  50 & 48 & $6 A_{2} + 2 A_{3}$ & ${\frac {  \eta \left( 16\tau \right)   ^{6}  \eta \left( 12\tau \right)   ^{2}}{  \eta \left( 48\tau \right)   ^{2}}}$ & 3 \\ 
  51 & 48 & $5 A_{1} +  A_{2} + 2 A_{3} +  A_{5}$ & ${\frac {  \eta \left( 24\tau \right)   ^{5}\eta \left( 16\tau \right)   \eta \left( 12\tau \right)   ^{2} \mbox{}\eta \left( 8\tau \right) }{  \eta \left( 48\tau \right)   ^{3}}}$ & 3 \\ 
  52 & 48 & $4 A_{1} + 3 A_{5}$ & ${\frac {  \eta \left( 24\tau \right)   ^{4}  \eta \left( 8\tau \right)   ^{3}}{  \eta \left( 48\tau \right)   ^{2}}}$ & 5/2 \\ 
  53 & 48 & $ A_{1} +  A_{2} + 2 A_{3} + 2 D_{5}$ & ${\frac {  \eta \left( 24\tau \right)   ^{5}  \eta \left( 16\tau \right)   ^{3}  \eta \left( 12\tau \right)   ^{2} \mbox{}  \eta \left( 4\tau \right)   ^{2}}{  \eta \left( 48\tau \right)   ^{3}  \eta \left( 8\tau \right)   ^{4}}}$ & 5/2 \\ 
  54 & 48 & $4 A_{1} +  A_{2} +  A_{7} +  E_{6}$ & ${\frac {  \eta \left( 24\tau \right)   ^{5}  \eta \left( 16\tau \right)   ^{3}\eta \left( 6\tau \right)  \mbox{}\eta \left( 2\tau \right) }{  \eta \left( 48\tau \right)   ^{3}  \eta \left( 4\tau \right)   ^{2}}}$ & 5/2 \\ 
  55 & 60 & $4 A_{1} + 3 A_{2} + 2 A_{4}$ & ${\frac {  \eta \left( 30\tau \right)   ^{4}  \eta \left( 20\tau \right)   ^{3}  \eta \left( 12\tau \right)   ^{2} \mbox{}}{  \eta \left( 60\tau \right)   ^{3}}}$ & 3 \\ 
  56 & 64 & $5 A_{1} + 3 A_{3} +  D_{4}$ & ${\frac {  \eta \left( 32\tau \right)   ^{8}\eta \left( 16\tau \right) \eta \left( 8\tau \right) }{  \eta \left( 64 \mbox{}\tau \right)   ^{4}}}$ & 3 \\ 
  57 & 64 & $6 A_{1} + 3 D_{4}$ & ${\frac {  \eta \left( 32\tau \right)   ^{15}  \eta \left( 8\tau \right)   ^{3}}{  \eta \left( 64\tau \right)   ^{6} \mbox{}  \eta \left( 16\tau \right)   ^{6}}}$ & 3 \\ 
  58 & 64 & $3 A_{1} + 3 A_{3} +  A_{7}$ & ${\frac {  \eta \left( 32\tau \right)   ^{3}  \eta \left( 16\tau \right)   ^{3}\eta \left( 8\tau \right)  \mbox{}}{  \eta \left( 64\tau \right)   ^{2}}}$ & 5/2 \\ 
  59 & 64 & $5 A_{3} +  D_{4}$ & ${\frac {  \eta \left( 32\tau \right)   ^{3}  \eta \left( 16\tau \right)   ^{3}\eta \left( 8\tau \right)  \mbox{}}{  \eta \left( 64\tau \right)   ^{2}}}$ & 5/2 \\ 
  60 & 64 & $4 A_{1} +  A_{3} + 2 D_{6}$ & ${\frac {  \eta \left( 32\tau \right)   ^{8}  \eta \left( 16\tau \right)   ^{3}  \eta \left( 4\tau \right)   ^{2} \mbox{}}{  \eta \left( 64\tau \right)   ^{4}  \eta \left( 8\tau \right)   ^{4}}}$ & 5/2 \\ 
  61 & 72 & $4 A_{1} + 3 A_{2} +  A_{3} +  D_{5}$ & ${\frac {  \eta \left( 36\tau \right)   ^{6}  \eta \left( 24\tau \right)   ^{4}\eta \left( 18\tau \right)  \mbox{}\eta \left( 6\tau \right) }{  \eta \left( 72\tau \right)   ^{4}  \eta \left( 12\tau \right)   ^{2}}}$ & 3 \\ 
  62 & 72 & $3 A_{1} + 2 A_{3} + 2 A_{5}$ & ${\frac {  \eta \left( 36\tau \right)   ^{3}  \eta \left( 18\tau \right)   ^{2}  \eta \left( 12\tau \right)   ^{2} \mbox{}}{  \eta \left( 72\tau \right)   ^{2}}}$ & 5/2 \\ 
  63 & 72 & $ A_{2} + 3 A_{3} + 2 D_{4}$ & ${\frac {\eta \left( 24\tau \right)   \eta \left( 9\tau \right)   ^{2}  \eta \left( 36\tau \right)   ^{6}}{  \eta \left( 72\tau \right)   ^{3} \mbox{}\eta \left( 18\tau \right) }}$ & 5/2 \\ 
  64 & 80 & $3 A_{1} + 4 A_{4}$ & ${\frac {  \eta \left( 40\tau \right)   ^{3}  \eta \left( 16\tau \right)   ^{4}}{  \eta \left( 80\tau \right)   ^{2}}}$ & 5/2 \\ 
  65 & 96 & $3 A_{1} + 3 A_{2} + 3 A_{3}$ & ${\frac {  \eta \left( 48\tau \right)   ^{3}  \eta \left( 32\tau \right)   ^{3}  \eta \left( 24\tau \right)   ^{3} \mbox{}}{  \eta \left( 96\tau \right)   ^{3}}}$ & 3 \\ 
  66 & 96 & $2 A_{1} + 2 A_{2} +  A_{3} + 2 A_{5}$ & ${\frac {  \eta \left( 48\tau \right)   ^{2}  \eta \left( 32\tau \right)   ^{2}\eta \left( 24\tau \right)  \mbox{}  \eta \left( 16\tau \right)   ^{2}}{  \eta \left( 96\tau \right)   ^{2}}}$ & 5/2 \\ 
  67 & 96 & $2 A_{1} + 3 A_{2} +  A_{7} +  D_{4}$ & ${\frac {  \eta \left( 48\tau \right)   ^{5}  \eta \left( 32\tau \right)   ^{3}  \eta \left( 12\tau \right)   ^{2} \mbox{}}{  \eta \left( 96\tau \right)   ^{3}  \eta \left( 24\tau \right)   ^{2}}}$ & 5/2 \\ 
  68 & 96 & $3 A_{1} + 2 A_{3} +  A_{5} +  D_{5}$ & ${\frac {  \eta \left( 48\tau \right)   ^{5}  \eta \left( 24\tau \right)   ^{2}\eta \left( 8\tau \right)  \mbox{}\eta \left( 32\tau \right) }{  \eta \left( 96\tau \right)   ^{3}\eta \left( 16\tau \right) }}$ & 5/2 \\ 
  69 & 96 & $3 A_{1} + 2 A_{2} + 2 E_{6}$ & ${\frac {  \eta \left( 48\tau \right)   ^{5}  \eta \left( 32\tau \right)   ^{6}  \eta \left( 4\tau \right)   ^{2} \mbox{}}{  \eta \left( 96\tau \right)   ^{4}  \eta \left( 8\tau \right)   ^{4}}}$ & 5/2 \\ 
  70 & 120 & $2 A_{1} +  A_{2} + 2 A_{3} +  A_{4} +  A_{5}$ & ${\frac {  \eta \left( 60\tau \right)   ^{2}\eta \left( 40\tau \right)   \eta \left( 30\tau \right)   ^{2} \mbox{}\eta \left( 24\tau \right) \eta \left( 20\tau \right) }{  \eta \left( 120\tau \right)   ^{2}}}$ & 5/2 \\ 
  71 & 128 & $3 A_{1} + 2 A_{3} +  D_{4} +  D_{6}$ & ${\frac {  \eta \left( 64\tau \right)   ^{8}\eta \left( 32\tau \right) \eta \left( 8\tau \right) }{  \eta \left( 128\tau \right)   ^{4} \mbox{}\eta \left( 16\tau \right) }}$ & 5/2 \\ 
  72 & 144 & $ A_{1} + 4 A_{2} + 2 A_{5}$ & ${\frac {\eta \left( 72\tau \right)   \eta \left( 48\tau \right)   ^{4}  \eta \left( 24\tau \right)   ^{2}}{  \eta \left( 144 \mbox{}\tau \right)   ^{2}}}$ & 5/2 \\ 
  73 & 160 & $2 A_{1} + 3 A_{3} + 2 A_{4}$ & ${\frac {  \eta \left( 80\tau \right)   ^{2}  \eta \left( 40\tau \right)   ^{3}  \eta \left( 32\tau \right)   ^{2} \mbox{}}{  \eta \left( 160\tau \right)   ^{2}}}$ & 5/2 \\ 
  74 & 168 & $ A_{1} + 3 A_{2} + 2 A_{3} +  A_{6}$ & ${\frac {\eta \left( 84\tau \right)   \eta \left( 56\tau \right)   ^{3}  \eta \left( 42\tau \right)   ^{2} \mbox{}\eta \left( 24\tau \right) }{  \eta \left( 168\tau \right)   ^{2}}}$ & 5/2 \\ 
  75 & 192 & $2 A_{1} + 6 A_{2} +  D_{4}$ & ${\frac {  \eta \left( 96\tau \right)   ^{5}  \eta \left( 64\tau \right)   ^{6}\eta \left( 24\tau \right)  \mbox{}}{  \eta \left( 192\tau \right)   ^{4}  \eta \left( 48\tau \right)   ^{2}}}$ & 3 \\ 
  76 & 192 & $2 A_{1} +  A_{2} + 2 A_{3} +  A_{5} +  D_{4}$ & ${\frac {  \eta \left( 96\tau \right)   ^{5}\eta \left( 64\tau \right) \eta \left( 32\tau \right)  \mbox{}\eta \left( 24\tau \right) }{  \eta \left( 192\tau \right)   ^{3}}}$ & 5/2 \\ 
  77 & 192 & $2 A_{1} +  A_{2} + 3 A_{3} +  E_{6}$ & ${\frac {  \eta \left( 96\tau \right)   ^{3}  \eta \left( 64\tau \right)   ^{3}  \eta \left( 48\tau \right)   ^{3} \mbox{}\eta \left( 8\tau \right) }{  \eta \left( 192\tau \right)   ^{3}  \eta \left( 16\tau \right)   ^{2}}}$ & 5/2 \\ 
  78 & 288 & $2 A_{1} + 2 A_{2} +  A_{3} + 2 D_{5}$ & ${\frac {  \eta \left( 144\tau \right)   ^{6}  \eta \left( 96\tau \right)   ^{4} \mbox{}\eta \left( 72\tau \right)   \eta \left( 24\tau \right)   ^{2}}{  \eta \left( 288\tau \right)   ^{4}  \eta \left( 48\tau \right)   ^{4}}}$ & 5/2 \\ 
  79 & 360 & $ A_{1} + 2 A_{2} + 2 A_{3} + 2 A_{4}$ & ${\frac {\eta \left( 180\tau \right)   \eta \left( 120\tau \right)   ^{2}  \eta \left( 90\tau \right)   ^{2} \mbox{}  \eta \left( 72\tau \right)   ^{2}}{  \eta \left( 360\tau \right)   ^{2}}}$ & 5/2 \\ 
  80 & 384 & $ A_{1} + 3 A_{2} + 2 A_{3} +  D_{6}$ & ${\frac {  \eta \left( 192\tau \right)   ^{3}  \eta \left( 128\tau \right)   ^{3} \mbox{}  \eta \left( 96\tau \right)   ^{3}\eta \left( 24\tau \right) }{  \eta \left( 384\tau \right)   ^{3}  \eta \left( 48\tau \right)   ^{2}}}$ & 5/2 \\ 
  81 & 960 & $ A_{1} + 3 A_{2} + 2 A_{4} +  D_{4}$ & ${\frac {  \eta \left( 480\tau \right)   ^{4}  \eta \left( 320\tau \right)   ^{3} \mbox{}  \eta \left( 192\tau \right)   ^{2}\eta \left( 120\tau \right) }{  \eta \left( 960\tau \right)   ^{3}  \eta \left( 240\tau \right)   ^{2}}}$ & 5/2 \\ 
   \hline
\caption{Table of the modular forms $Z_{X,G}^{-1}$ for all symplectic
$G$ actions.} \label{table: list of eta products}
\end{longtable}


\ifx\undefined\bysame
\newcommand{\bysame}{\leavevmode\hbox to3em{\hrulefill}\,}
\fi

\end{document}